\documentclass{article}
\usepackage[utf8]{inputenc}

\usepackage{amssymb,latexsym,amsmath,epsfig,amsthm,amsmath,amscd, graphicx, verbatim, setspace,authblk} 

\makeatletter

\renewcommand\section{\@startsection {section}{1}{\z@}
{-30pt \@plus -1ex \@minus -.2ex}
{2.3ex \@plus.2ex}
{\normalfont\normalsize\bfseries\boldmath}}

\renewcommand\subsection{\@startsection{subsection}{2}{\z@}
{-3.25ex\@plus -1ex \@minus -.2ex}
{1.5ex \@plus .2ex}
{\normalfont\normalsize\bfseries\boldmath}}

\renewcommand{\@seccntformat}[1]{\csname the#1\endcsname. }

\makeatother

\theoremstyle{definition}
\theoremstyle{plain}
\newtheorem{thm}{Theorem}[section]
\newtheorem{lem}[thm]{Lemma}
\newtheorem{cor}[thm]{Corollary}
\newtheorem{conj}[thm]{Conjecture}

\newtheorem{definition}[thm]{Definition}

\usepackage{mathtools}
\usepackage{amsthm}
\usepackage{amssymb}
\usepackage{amsfonts}
\usepackage{graphicx}
\usepackage{ytableau}
\usepackage{hyperref}
\usepackage[utf8]{inputenc}
\usepackage{todonotes}
\usepackage{tikz}
\usepackage{comment}

\newcommand{\exm}{\operatorname{ex}}
\newcommand{\satm}{\operatorname{sat}}
\newcommand{\ssatm}{\operatorname{ssat}}

\newcommand{\ex}{\operatorname{Ex}}
\newcommand{\sat}{\operatorname{Sat}}
\newcommand{\ssat}{\operatorname{Ssat}}

\newcommand{\up}{\operatorname{up}}

\title{Sequence saturation}
\author{Anand, Jesse Geneson, Suchir Kaustav, and Shen-Fu Tsai}
\date{}

\begin{document}

\maketitle

\begin{abstract}
In this paper, we introduce saturation and semisaturation functions of sequences, and we prove a number of fundamental results about these functions. Given a forbidden sequence $u$ with $r$ distinct letters, we say that a sequence $s$ on a given alphabet is $u$-saturated if $s$ is $r$-sparse, $u$-free, and adding any letter from the alphabet to an arbitrary position in $s$ violates $r$-sparsity or induces a copy of $u$. We say that $s$ is $u$-semisaturated if $s$ is $r$-sparse and adding any letter from the alphabet to $s$ violates $r$-sparsity or induces a new copy of $u$. Let the saturation function $\operatorname{Sat}(u, n)$ denote the minimum possible length of a $u$-saturated sequence on an alphabet of size $n$, and let the semisaturation function $\operatorname{Ssat}(u, n)$ denote the minimum possible length of a $u$-semisaturated sequence on an alphabet of size $n$. For alternating sequences, we determine both the saturation function and the semisaturation function up to a constant multiplicative factor. We show for every sequence that the semisaturation function is always either $O(1)$ or $\Theta(n)$. For the saturation function, we show that every sequence $u$ has either $\operatorname{Sat}(u, n) \ge n$ or $\operatorname{Sat}(u, n) = O(1)$. For every sequence with $2$ distinct letters, we show that the saturation function is always either $O(1)$ or $\Theta(n)$.
\end{abstract}

\section{Introduction}

We say that sequences $u$ and $v$ are isomorphic if $v$ can be obtained from $u$ by a one-to-one renaming of the letters. Given a forbidden sequence $u$, we say that the sequence $s$ avoids $u$ if no subsequence of $s$ is isomorphic to $u$. If $s$ avoids $u$, then we call $s$ $u$-free. We say that $s$ is $r$-sparse if any $r$ consecutive letters in $s$ are distinct. For example, $s$ is $2$-sparse if it has no adjacent same letters. For another example, the sequence $a b a b \dots$ is $2$-sparse but not $3$-sparse, and the sequence $a b c a b c \dots$ is $3$-sparse but not $4$-sparse. 

In this paper, we focus on the lengths of sequences $s$ which avoid some forbidden sequence $u$ and satisfy a notion of maximality. If $u$ has a total of $r$ distinct letters, then we require that $s$ must be $r$-sparse. Otherwise, it is clear that $s$ can be arbitrarily long and still avoid $u$. 

Davenport-Schinzel sequences of order $s$ are sequences with no adjacent same letters and no alternating subsequences of length $s+2$. In other words, for any integer $s\ge 0$, let sequence $u_s:=xyxy\dots$ be an alternating sequence of $2$ letters $x$ and $y$ of length $s+2$. An $(n,s)$ Davenport-Schinzel sequence, or $DS(n,s)$-sequence, is a $2$-sparse sequence on $n$ letters that has no subsequence isomorphic to $u_s$. Davenport-Schinzel sequences have been used for a number of applications including bounds on the complexity of lower envelopes \cite{ds} and the complexity of faces in arrangements of arcs \cite{agarwal}. 

Generalized Davenport-Schinzel sequences are sequences which avoid some forbidden sequence. More precisely, given any forbidden sequence $u$ with a total of $r$ distinct letters, the generalized Davenport-Schinzel sequences corresponding to $u$ are $r$-sparse sequences which avoid $u$. Generalized Davenport-Schinzel sequences have also been used for a number of applications including bounds on the number of edges in $k$-quasiplanar graphs \cite{FPS, SW} and the maximum number of ones in 0-1 matrices which avoid certain forbidden patterns \cite{cibulka, geneson20, geneson19, geneson21a, wellman}. 

In particular, it is possible to translate many results about generalized Davenport-Schinzel sequences into results about 0-1 matrices with forbidden patterns \cite{keszegh} and vice versa \cite{pettie11}. In addition, connections have been found between generalized Davenport-Schinzel sequences and collections of interval chains that are not stabbed by the same tuple \cite{geneson15}. 

Most of the focus on generalized Davenport-Schinzel sequences has been on their extremal functions. Specifically, given a forbidden sequence $u$ with $r$ distinct letters, we define $\ex(u, n)$ to be the maximum possible length of an $r$-sparse sequence with at most $n$ distinct letters which avoids $u$. Nivasch \cite{nivasch} and Pettie \cite{pettie15} found nearly sharp bounds on $\ex(u, n)$ for all alternations $u$, i.e., the maximum possible lengths of Davenport-Schinzel sequences of all orders. Bounds for forbidden sequences with more distinct letters such as $(a_1 a_2 \dots a_r)^t$ and $abc (acb)^t abc$ were found in \cite{GPT, GT15, GTT} using computational methods combined with bounds from \cite{nivasch} on extremal functions of special families of sequences called formations. 

In this paper, we focus on $2$ different but related functions for generalized Davenport-Schinzel sequences called saturation functions and semisaturation functions. Saturation functions have previously been studied for graphs \cite{damasdi, erdos, furedi, korandi}, ordered graphs \cite{BK}, posets \cite{ferrera, KLM}, set systems \cite{frankl, gerbner}, 0-1 matrices \cite{BC, FK, geneson21, berendsohn}, and multi-dimensional 0-1 matrices \cite{gt23, tsai23}. To motivate the definitions of saturation and semisaturation functions for sequences, we discuss the definitions and some results for saturation and semisaturation functions of 0-1 matrices. 

\subsection{Pattern avoidance, saturation, and semisaturation for 0-1 matrices}

We say that 0-1 matrix $A$ avoids 0-1 matrix $B$ if no submatrix of $A$ can be turned into $B$ by changing any number of ones to zeroes. The extremal function $\exm(n, P)$ is the maximum number of ones in an $n \times n$ 0-1 matrix which avoids $P$. The study of this extremal function has been motivated by several problems that may at first seem unrelated such as the Stanley-Wilf conjecture \cite{MT}, the number of unit distances in a convex $n$-gon \cite{ngon}, and algorithms for finding minimal paths in rectilinear grids with obstacles \cite{mitchell}.

Much of the focus on extremal functions of 0-1 matrices has been on determining when $\exm(n, P)$ is linear and when it is nonlinear. Marcus and Tardos \cite{MT} showed that permutation matrices have linear extremal functions, thus implying the Stanley-Wilf conjecture by a result of Klazar \cite{klazar}. This linear upper bound was later generalized to any matrix obtained from a permutation matrix by doubling the columns \cite{geneson2009}. In combination with a construction from \cite{keszegh}, this fact was used to demonstrate the existence of infinitely many 0-1 matrices which are minimally nonlinear \cite{geneson2009}. Fox \cite{fox13} sharpened the upper bounds on permutation matrices, and these bounds were extended to multidimensional permutation matrices \cite{KM, GT17}. Other papers \cite{Crowdmath2018, GT2020} identified a number of properties of minimally non-linear 0-1 matrices, including bounds on the ratio of their dimensions.

Brualdi and Cao \cite{BC} initiated the study of saturation functions of 0-1 matrices. They defined the saturation function $\satm(n, P)$ to be the minimum number of ones in an $n \times n$ 0-1 matrix which avoids $P$ such that changing any zero to a one induces a copy of $P$. Fulek and Keszegh \cite{FK} found a dichotomy for the 0-1 matrix saturation function: it is always either $\Theta(n)$ or $O(1)$. They found a single permutation matrix $P$ for which $\satm(n, P) = O(1)$ and asked if there are more. Geneson \cite{geneson21} showed that almost all permutation matrices $P$ have $\satm(n, P) = O(1)$, and Berendsohn \cite{berendsohn} later characterized when $\satm(n, P) = O(1)$ for permutation matrices $P$. It is still an open problem to characterize when $\satm(n, P) = O(1)$ for general 0-1 matrices $P$.

In addition to their results for saturation functions of 0-1 matrices, Fulek and Keszegh also defined the semisaturation function $\ssatm(n, P)$ to be the minimum number of ones in an $n \times n$ 0-1 matrix such that changing any zero to a one induces a new copy of $P$. Note that in this definition, the $n \times n$ 0-1 matrix does not have to avoid $P$. By definition, we have $\ssatm(n, P) \le \satm(n, P) \le \exm(n, P)$. As with the saturation function of 0-1 matrices, Fulek and Keszegh found that the semisaturation function is always either $\Theta(n)$ or $O(1)$. Unlike the saturation function, they exactly characterized the dichotomy for the semisaturation function. Tsai \cite{tsai23} extended the saturation and semisaturation functions to multidimensional 0-1 matrices.

\subsection{Saturation functions and semisaturation functions of forbidden sequences}

In analogy with forbidden 0-1 matrices, we define saturation and semisaturation functions of forbidden sequences. Given a forbidden sequence $u$ with $r$ distinct letters, we say that a sequence $s$ on a given alphabet is $u$-saturated if $s$ is $r$-sparse, $u$-free, and adding any letter from the alphabet to an arbitrary position in $s$ violates $r$-sparsity or induces a copy of $u$. Let the saturation function $\sat(u, n)$ denote the minimum possible length of a $u$-saturated sequence on an alphabet of size $n$. It is easy to see that $\sat(u, n) = \min(n, k-1)$ and $\sat(v,n)=n$ for $n\geq 1$ when $u = a_1a_2\ldots a_k$ and $v=a_1a_2\ldots a_ka_1$. In the special case that $u$ is an alternation, we let $\xi (n, s)$ denote the minimum possible length of a $u_s$-saturated sequence on an alphabet of size $n$. Recall that $u_s$ has length $s+2$, and note that $\xi(n,1)=n$ for every $n \ge 1$.

Given a forbidden sequence $u$ with $r$ distinct letters, we say that a sequence $s$ on a given alphabet is $u$-semisaturated if $s$ is $r$-sparse and adding any letter from the alphabet to $s$ violates $r$-sparsity or induces a new copy of $u$. Note that with sequence semisaturation, we do not require $s$ to avoid $u$, as with 0-1 matrix semisaturation. Let the semisaturation function $\ssat(u, n)$ denote the minimum possible length of a $u$-semisaturated sequence on an alphabet of size $n$. Since we require $r$-sparsity in both definitions, we have $\ssat(u, n) \le \sat(u, n) \le \ex(u, n)$. As with the sequence saturation function, it is easy to see that $\ssat(u, n) = \min(n, k-1)$ and $\ssat(v,n)=n$ for $n\geq 1$ when $u = a_1a_2\ldots a_k$ and $v=a_1a_2\ldots a_ka_1$.

\subsection{Our results}

In this paper, we prove a number of results about saturation functions and semisaturation functions of sequences. For alternating sequences of the form $a b a b \dots$, we determine both the saturation function and the semisaturation function up to a constant multiplicative factor. We also prove a number of structural results about sequences which are saturated for alternations. For example, we show that if an alternating sequence $u_s$ has even length, then the first and last letters of any $u_s$-saturated sequence are the same. On the other hand, we show that if $u_s$ has odd length, then the first and last letters of any $u_s$-saturated sequence are different. 

We also prove general results about saturation and semisaturation functions. In particular, we demonstrate a dichotomy for the semisaturation functions of sequences. Specifically, we show that for any sequence $u$ we have $\ssat(u, n) = O(1)$ if and only if the first letter and the last letter of $u$ each occur exactly once, and otherwise we have $\ssat(u, n) = \Theta(n)$. For the saturation function, we show that every sequence $u$ has either $\sat(u, n) \ge n$ for every positive integer $n$ or $\sat(u, n) = O(1)$. We prove that every sequence $u$ in which every letter occurs at least twice has $\sat(u, n) \ge n$. Moreover, we show that $\sat(u, n) = \Theta(n)$ or $\sat(u, n) = O(1)$ for every sequence $u$ with $2$ distinct letters. 

Our results about saturation functions of alternating sequences are in Section~\ref{s:sat_alt}. In Section~\ref{s:sat_gen}, we prove the general results about saturation functions of sequences. Our results about semisaturation functions of sequences are in Section~\ref{s:ssat}. Finally, in Section~\ref{s:conc}, we discuss some remaining questions and future directions.

\section{Saturation functions of Davenport-Schinzel sequences}\label{s:sat_alt}

We start with a construction which yields an upper bound on the saturation functions of alternating sequences.

\begin{lem}
	\label{lemma:construction}
$\xi(n,s)\leq s(n-1)+1$ for all $n\ge 2$.
\end{lem}

\begin{proof}
Use the following construction
$$\begin{cases}
[\underbrace{1,2,1,2,\ldots, 1,2}_{s\text{~terms}},\underbrace{1,3,1,3,\ldots, 1,3}_{s\text{~terms}},\ldots, \underbrace{1,n,1,n,\ldots, 1,n}_{s\text{~terms}},1] & \text{if $s$ is even}\\
[\underbrace{1,2,1,2,\ldots,1}_{s\text{~terms}},\underbrace{2,3,2,3,\ldots,2}_{s\text{~terms}},\ldots, \underbrace{n-1,n,n-1,n,\ldots,n-1}_{s\text{~terms}},n] & \text{if $s$ is odd}
\end{cases}$$
Clearly, the above $DS(n,s)$-sequences of length $s(n-1)+1$ are $u_s$-saturated.
\end{proof}

In the next lemma, we prove several results about $a b a b$-saturated sequences. Note that $u = a b a b$ is the longest alternating sequence for which $\ex(u, n) = O(n)$ \cite{hs86}.

\begin{lem}
	For any $abab$-saturated sequence $x=(x_1,x_2,\dots)$ on $n$ letters
\begin{enumerate}
\item The length of $x$ is $2n-1$.
\item The first and last letter of $x$ are identical.
\item Suppose $x_i=x_j=a_k$ for some $i<j$ and there is no other occurrence of $a_k$ between them. Then $x_{i+1}=x_{j-1}$.
\end{enumerate}
\end{lem}
\begin{proof}
\begin{enumerate}
\item 
The proof of Lemma~\ref{lemma:construction} shows the existence of $abab$-saturated sequence on $n$ letters with length $2n-1$.
We prove the result by induction on $n$.
The statement holds for $n=1$, so assume $n>1$. Consider an $abab$-saturated sequence $x$ on $n$ letters whose first letter is $1$. Sequence $x$ must have at least another $1$, because otherwise appending a $1$ to the end of $x$ does not create a subsequence isomorphic to $abab$. Let the subsequence of $x$ between the first $2$ occurrences of $1$ be $u$ and write
$$x=1u1v.$$
Observe that the letters in sequences $u$ and $v$ are disjoint, as otherwise $x$ contains $abab$. Note that every letter in $[n]$ appears in $x$, as otherwise adding the letter to $x$ does not create a copy of $abab$. Assume that $u$ has $k$ distinct letters where $0<k<n$, so sequence $1v$ has exactly $n-k$ distinct letters. Sequence $u$ itself is $abab$-saturated, hence by inductive hypothesis its length is $2k-1$. Similarly sequence $1v$ is $abab$-saturated and therefore has length $2(n-k)-1$. Thus, the total length of $x$ is $2k-1+2(n-k)-1+1=2n-1$.
\item Continuing the inductive proof,  sequence $1v$ is $abab$-saturated. 
By inductive hypothesis the last letter of $1v$ is $1$, i.e., the first and last letters of $x$ are identical.
	\item Similarly, by inductive hypothesis the statement holds for $u$ and $1v$. Therefore, we only need to check the first $2$ occurrences of $1$ in $x$. The statement still holds because by inductive hypothesis the first and last letters of $u$ are identical.
	\end{enumerate}
\end{proof}

In the next result, we prove a lower bound on saturation functions of alternating sequences which is within a multiplicative factor of $2$ of our upper bound.

\begin{lem}\label{lem:altlower}
For all $n \ge 3$ and $s \ge 1$, we have $\xi(n,s)\ge n\lfloor s/2\rfloor+1$ if $s$ is even and $\xi(n,s)\ge n\lfloor s/2\rfloor+3$ if $s$ is odd.
\end{lem}
\begin{proof}
We first consider the case when there exists some letter $a_1$ such that adding it to a $u_s$-saturated sequence $x$ violates $2$-sparsity. Clearly $x$ has odd length and every odd-positioned letter of $x$ is $a_1$. Letter $a_1$ has at least $2$ more occurrences than any other letter, since otherwise all even-positioned letters are identical and $x$ has only $2$ distinct letters. Every even-positioned letter can be appended to $x$ without violating $2$-sparsity. If $s$ is even, then every even-positioned letter has at least $s/2$ occurrences and $a_1$ has at least $(s+4)/2$ occurrences. This implies the length of $x$ is at least $2+ns/2$. If $s$ is odd, then every even-positioned letter has at least $(s+1)/2$ occurrences and $a_1$ has at least $(s+5)/2$ occurrences, so $x$ has length at least $ns/2+(n+4)/2$.

    Now, we show that if every letter in the alphabet can be added to a $u_s$-saturated sequence $x$ without violating $2$-sparsity, then each letter appears at least $\lfloor s/2\rfloor$ times. Suppose that some letter $i$ appears at most $\lfloor s/2\rfloor-1$ times. If we add a new $i$, it has at most $\lfloor s/2\rfloor$ occurrences. The longest alternating subsequence containing $i$ has length no more than $2\left(\lfloor s/2\rfloor\right)+1\le s+1<s+2$, a contradiction.
    
    If $s$ is even, suppose that every letter appears only $\lfloor s/2\rfloor$ times. Only one letter appears $(s+2)/2$ times after adding a letter while both letters in $u_s$ appear $(s+2)/2$ times. So, some letter appears $s/2+1$ times and $x$ has length at least $ns/2+1$. If $s$ is odd and if every letter appears at least $\lfloor (s+2)/2\rfloor$ times, then the length of $x$ is at least $n\lfloor (s+2)/2\rfloor\ge n\lfloor s/2\rfloor+3$. If some letter $a_1$ appears only $\lfloor s/2\rfloor$ times, then it requires some letter $a_2$ to appear at least $(s+3)/2=\lfloor s/2\rfloor+2$ times, which also requires some letter $a_3$ to appear at least $(s+1)/2=\lfloor s/2 \rfloor+1$ times. That is, the length of $x$ is at least $n\lfloor s/2\rfloor+3$.
\end{proof}

In the next several lemmas, we prove a number of structural results about sequences that are saturated for alternations. We start with a useful definition.

\begin{definition}
	In a sequence $x$ on $n$ letters that is $u_s$-saturated, distinct letters $i$ and $j$ are called \textbf{friends} if the longest alternating subsequence formed by $i$ and $j$ has length $s$ or $s+1$.
\end{definition}

\begin{lem}\label{lem:friends}
	In a sequence $x$ on $n$ letters that is $u_s$-saturated, there are no adjacent letters that are not friends.
\end{lem}
\begin{proof}
	Suppose that in $x$ that adjacent letters $j$ and $k$ are not friends. Without loss of generality assume that $j$ occurs first in the adjacent pair and that either $k$ is the last letter of $x$ or the letter after $k$ is not $j$. We can then insert or append a $j$ after $k$ to obtain a new sequence $x'$, and a subsequence $u_s'$ of $x'$ is isomorphic to $u_s$. However $u_s'$ cannot consist of $j$ and $k$ because the longest alternating subsequence on $j$ and $k$ in $x'$ has length no more than $s-1+2=s+1$. Therefore, $u_s'$ consists of $j$ and $j'\ne k$, and since the $k$ between the added $j$ and its preceding $j$ can be ignored, $x$ already contains $u_s$.
\end{proof}

\begin{lem}
	\label{lemma:alt}
	Suppose that in a $u_s$-saturated sequence $x$ on $n$ letters, some letter $a_1$ is followed by a letter $a_2$, which is not followed by $a_1$. Then adding an $a_1$ right after $a_2$ creates a length-$(s+2)$ alternating subsequence composed of $a_1$ and $a_2$.
\end{lem}
\begin{proof}
If the addition of $a_1$ right after $a_2$ does not create a length-$(s+2)$ alternating subsequence composed of $a_1$ and $a_2$, then it creates a length-$(s+2)$ alternating subsequence composed of $a_1$ and $a_3$ for some letter $a_3\ne a_2$. Moreover this subsequence cannot contain both the original and the added $a_1$, so $x$ already contained $u_s$ before adding $a_1$.
\end{proof}

In the next lemma and its resulting corollaries, we investigate the structure of $u_s$-saturated sequences between $2$ occurrences of the same letter that have no other occurrences of the letter between them. 

\begin{lem}
	\label{lemma:disjoint}
	Suppose that in a $u_s$-saturated sequence $x=(x_1,x_2,\dots)$ on $n$ letters, $x_{i_1}=x_{i_2}=a_1$ are identical letters such that $i_1 < i_2$, and there is no other $a_1$ between $x_{i_1}$ and $x_{i_2}$. Then the subsequence $u$ between $x_{i_1}$ and $x_{i_2}$ cannot be written as $u=vw$ where $v$ and $w$ are nonempty sequences such that their letters are disjoint.
\end{lem}
\begin{proof}
	Suppose $u$ can be written as $u=vw$ where $v$ and $w$ are nonempty sequences such that their letters are disjoint. Inserting an $a_1$ between $v$ and $w$ creates a copy of $u_s$, but the letters of $v$ and $w$ are disjoint, so $x$ already contained $u_s$ before inserting $a_1$.
\end{proof}

\begin{cor}
	\label{corollary:another}
	Suppose that in a $u_s$-saturated sequence $x=(x_1,x_2,\dots)$ on $n$ letters, $x_{i_1}=x_{i_2}=a_1$ are identical letters such that $i_1+3 \le i_2$, and there is no other $a_1$ between $x_{i_1}$ and $x_{i_2}$. Then $i_1+4\le i_2$ and there is another occurrence of $a_2=x_{i_1+1}$ at some position in $[i_1+3,i_2-1]$.
\end{cor}
\begin{proof}
Express the subsequence between $x_{i_1}$ and $x_{i_2}$ as $a_2w$ where $a_2$ is a letter and $w$ is a sequence. By Lemma~\ref{lemma:disjoint}, $w$ contains $a_2$, which, by sparsity, is not the first letter of of $w$.
\end{proof}

\begin{cor}
	Suppose that in a $u_s$-saturated sequence $x=(x_1,x_2,\dots)$ on $n$ letters, $x_{i_1}=a_1$ is the first occurrence of $a_1$ and $x_{i_1+1}=a_2\ne a_1$. Then either $x_{i_1+1}$ is the first occurrence of $a_2$ or $x_{i_1-1}=a_2$.
\end{cor}
\begin{proof}
	Suppose $x_{i_1+1}$ is not the first occurrence of $a_2$ and $x_{i_1-1}\ne a_2$. Then, assume the previous occurrence of $a_2$ is $x_{i_0}$, so we have $i_0+3\le i_1+1$. By applying Corollary~\ref{corollary:another} to the reverse of $x$, there is another occurrence of $a_1$ between $x_{i_0}$ and $x_{i_1}$, i.e., $x_{i_1}$ is not the first occurrence of $a_1$, contradicting our assumption.
\end{proof}

In the next two lemmas, we focus on the first and last letters of $u_s$-saturated sequences. In particular, we prove that these letters are the same when $s$ is even, but different when $s$ is odd.

\begin{lem}
\label{lemma:even-same}
If $s$ is even, then the first and last letters of any $u_s$-saturated sequence are the same.
\end{lem}
\begin{proof}
Suppose that $x$ is $u_s$-saturated and the first letter of $x$ is $a$. For contradiction assume that the last letter of $x$ is not $a$, so we append an $a$ to $x$ to obtain a new sequence $x'$, which has an alternating subsequence $u'$ of length $s+2$. Subsequence $u'$ must include the appended $a$, and the first letter of $u'$ is not $a$ because $s$ is even. Thus, the first letter of $x$ concatenated with $u'$ except its last letter is an alternating subsequence of $x$ of length $s+2$. Therefore, $x$ contains $u_s$, a contradiction.
\end{proof}

\begin{lem}
If $s$ is odd, then the first and last letters of any $u_s$-saturated sequence are different.
\end{lem}
\begin{proof}
    Similarly to the proof of Lemma~\ref{lemma:even-same}, suppose $x$ is $u_s$-saturated with some odd $s$. Let the last and second to last letters of $x$ be $a$ and $b$, respectively. Assume for contradiction that the first letter of $x$ is $a$. Append $b$ to $x$ to obtain a new sequence $x'$, which contains an alternating subsequence $u'$ of length $s+2$. This is an alternating sequence of $a$ and $b$, because otherwise, we could turn this alternating sequence into an alternating sequence of length $s+2$ in $x$ by replacing the added $b$ by the $b$ at the second to last position on $x$. Because $s$ is odd, the first letter of $u'$ is $b$. The first letter of $x$ and the first $s+1$ letters of $u'$ form an alternating subsequence of length $s+2$, so $x$ contains $u_s$, a contradiction.
\end{proof}

\begin{lem}
    For $n\ne 2$ suppose $x$ is a $u_3$-saturated sequence and has some maximal contiguous alternating subsequence $aba$. Then $a$ has at least three occurrences in $x$.
\end{lem}
\begin{proof}
   Suppose this contiguous alternating subsequence is  $(x_i,x_{i+1},x_{i+2})=aba$.
   Insert a $b$ right after $x_{i+2}=a$ to obtain sequence $x'$, which contains $u_3$. Note that this copy of $u_3$ must be on the letters $a, b$ or else $x$ would have already contained $u_3$. 

If $a$ has only two ocurrences in $x$, that is, $x_i$ and $x_{i+2}$, then the added $b$ is the last letter of the copy of $u_3$, so there is an occurrence of $b$ before $x_i$. By a symmetric argument of adding $b$ befoe $x_i$, there is an occurrence of $b$ in $x$ after $x_{i+2}$. Then the $b$ before $x_i$ together with $x_i,x_{i+1},x_{i+2}$ and the $b$ after $x_{i+2}$ form an occurrence of $u_3$, a contradiction.
\end{proof}

Below we list some of our conjectures about saturation for alternating sequences. 

\begin{conj}\label{c:sat_alt}
Suppose $x=(x_1,x_2,\dots)$ is a shortest $u_s$-saturated sequence on $n$ letters.
\begin{enumerate}
\item $x$ has length $ns-s+1$.
\item If $s$ is odd, then for any distinct letters $i$ and $j$, $i$'s first occurrence precedes $j$'s first occurrence if and only if $i$'s last occurrence precedes $j$'s last occurrence. If we order $n$ letters as $0,1,\dots,n-1$ by their first occurrences, then letters $0$ and $n-1$ each have $\frac{s+1}{2}$ occurrences and all other letters have $s$ occurrences.
\item If for $i < j$, $x_i$ and $x_j$ are $2$ occurrences of letter $a_k$ with no other occurrence of $a_k$ between them, then $x_{i+1}=x_{j-1}$.
\end{enumerate}
\end{conj}

In the table below, we have used a program to generate some $u_s$-saturated sequences which we conjecture to have minimum length. \\

\begin{tabular}{|c|c|l|}
	\hline
	s & n & some $u_s$-saturated sequences likely to be the shortest\\
	\hline
	2 & 7 & 0, 1, 2, 1, 3, 1, 4, 5, 4, 1, 6, 1, 0 \\
	\hline
	2 & 7 & 0, 1, 0, 2, 0, 3, 4, 5, 4, 6, 4, 3, 0 \\
	\hline
	2 & 7 & 0, 1, 2, 1, 3, 1, 4, 1, 0, 5, 6, 5, 0 \\
	\hline
	3 & 7 & 0, 1, 2, 1, 0, 1, 2, 3, 4, 3, 2, 3, 4, 5, 6, 5, 4, 5, 6 \\
	\hline
	3 & 7 & 0, 1, 2, 1, 0, 1, 2, 3, 2, 3, 4, 5, 4, 3, 4, 5, 6, 5, 6 \\
	\hline
	3 & 7 & 0, 1, 0, 1, 2, 1, 2, 3, 4, 5, 4, 3, 2, 3, 4, 5, 6, 5, 6 \\
	\hline
	3 & 7 & 0, 1, 0, 1, 2, 1, 2, 3, 2, 3, 4, 3, 4, 5, 6, 5, 4, 5, 6 \\
	\hline
	4 & 7 & 0, 1, 0, 1, 0, 2, 0, 2, 0, 3, 4, 3, 0, 3, 4, 3, 5, 3, 6, 3, 5, 3, 6, 3, 0 \\
	\hline
	4 & 7 & 0, 1, 0, 1, 2, 1, 2, 1, 3, 1, 3, 1, 4, 1, 4, 1, 5, 1, 6, 1, 5, 1, 6, 1, 0 \\
	\hline
	4 & 7 & 0, 1, 0, 1, 0, 2, 0, 2, 3, 2, 3, 2, 0, 4, 0, 4, 0, 5, 0, 5, 0, 6, 0, 6, 0 \\ 
	\hline
	4 & 7 & 0, 1, 0, 1, 0, 2, 0, 2, 0, 3, 0, 4, 0, 5, 0, 6, 0, 3, 0, 4, 0, 5, 0, 6, 0 \\
	\hline
	4 & 7 & 0, 1, 0, 2, 0, 2, 0, 1, 0, 3, 0, 4, 0, 3, 5, 3, 5, 3, 6, 3, 6, 3, 0, 4, 0 \\
	\hline
	5 & 6 & 0, 1, 2, 1, 0, 1, 0, 1, 2, 3, 2, 1, 2, 3, 2, 3, 4, 5, 4, 3, 4, 5, 4, 3, 4, 5 \\
	\hline
	5 & 6 & 0, 1, 2, 1, 0, 1, 2, 1, 0, 1, 2, 3, 2, 3, 2, 3, 4, 5, 4, 3, 4, 5, 4, 3, 4, 5 \\
	\hline
	5 & 6 & 0, 1, 0, 1, 0, 1, 2, 1, 2, 3, 2, 1, 2, 3, 2, 3, 4, 3, 4, 3, 4, 5, 4, 5, 4, 5 \\
	\hline
	5 & 6 & 0, 1, 0, 1, 0, 1, 2, 1, 2, 1, 2, 3, 2, 3, 2, 3, 4, 3, 4, 3, 4, 5, 4, 5, 4, 5 \\
	\hline
\end{tabular}
\section{General results about saturation functions of sequences}\label{s:sat_gen}

We start with some simple results about the saturation function of generalized Davenport-Schinzel sequences.

\begin{lem}
    If $u$ has length $r$ and every letter occurs once in $u$, then $\sat(u, n) = \min(n, r-1)$. 
\end{lem}

\begin{proof}
    If $u$ has a total of $r$ letters, then any sequence of $\min(n,r-1)$ distinct letters is $u$-saturated, so $\sat(u, n) \le \min(n,r-1)$. Moreover, any $u$-saturated sequence must have at least $\min(n, r-1)$ distinct letters, since we can add a new letter anywhere to any $r$-sparse sequence with at most $\min(n, r-1)-1$ distinct letters and still avoid $u$. Thus, $\sat(u, n) = \min(n, r-1)$.
\end{proof}

\begin{lem}\label{lem:sat_first}
    Suppose that $u$ is a sequence in which the first letter occurs multiple times or the last letter occurs multiple times. Then $\sat(u, n) \ge n$.
\end{lem}

\begin{proof}
    Without loss of generality, suppose that the first letter of $u$ occurs multiple times. If $x$ is a $u$-saturated sequence on the letters $a_1, a_2, \dots, a_n$, then we claim that every letter $a_1, a_2, \dots, a_n$ must occur at least once in $x$. Otherwise, if $a_i$ is some letter that does not occur in $x$, then we can add $a_i$ to the front of $x$ and obtain a new sequence which still avoids $u$.
\end{proof}

\begin{cor}
    For all $n \ge 1$ and $k \ge 1$, we have $\sat(a_1 a_2 \dots a_k a_1, n) = n$. 
\end{cor}

\begin{proof}
    Let $s$ be an $a_1 a_2 \dots a_k a_1$-saturated sequence on an alphabet of $n$ letters. If some letter occurs at least twice in $s$, then $s$ would contain $a_1 a_2 \dots a_k a_1$ since $s$ is $k$-sparse. Thus, every letter in $s$ occurs at most once, so we have $\sat(a_1 a_2 \dots a_k a_1, n) \le n$.

    By Lemma~\ref{lem:sat_first}, we have $\sat(a_1 a_2 \dots a_k a_1, n) \ge n$. Thus, $\sat(a_1 a_2 \dots a_k a_1, n) = n$. 
\end{proof}

In the next result, we obtain a linear bound on the saturation functions of sequences of the form $(a_1a_2\ldots a_k)^r$. Interestingly enough, these same sequences play an important role in our lower bounds for semisaturation functions of alternating sequences in the next section. 

\begin{lem}\label{lem:up}
$\sat(u,n)\leq (kr-k)n-(k-1)(kr-k-1)$ for all $n\geq k$ where $u=(a_1a_2\ldots a_k)^r$.
\end{lem}
\begin{proof}
$$A_k:=[1,2,\ldots,k-1].$$
$$A_{i,k,r}:=[\underbrace{A_k,i,A_k,i,\ldots,A_k,i}_{r-1\text{~pairs of $A_k,i$}}].$$
Now consider the following construction
$$S_{u,n}:=[A_{k,k,r},A_{k+1,k,r},\ldots,A_{n,k,r},A_k].$$
For example, for $k=5$ and $r=3$
$$
S_{u,n}=[12345,12345,12346,12346,\dots ,1234n,1234n,1234].
$$
It is not difficult to see that this sequence is $k$-sparse over $n$ letters and has no subsequence isomorphic to $u$. Now, we show that adding a letter from $[n]$ to $S_{u,n}$ would either violate sparsity or create a subsequence isomorphic to $u$.
Adding a letter in $[k-1]$ would violate sparsity.
Adding $i\ge k$ inside $A_{i,k,r}$ or right before one of the next $k-1$ letters after $A_{i,k,r}$ would violate sparsity.
Adding $i\ge k$ elsewhere would create, together with $A_{i,k,r}$ and the next $k-1$ letters, an occurrence of $u$ on the letters $[k-1]\cup\{i\}$.
Sequence $S_{u,n}$ has length $(kr-k)n-(k-1)(kr-k-1)$, so our result follows.
\end{proof}

\begin{cor}
    For any $r \ge 2$ and $u = (a_1a_2\ldots a_k)^r$, we have $\sat(u, n) = \Theta(n)$.
\end{cor}

\begin{proof}
    The upper bound follows from Lemma~\ref{lem:up}, and the lower bound follows from Lemma~\ref{lem:sat_first}.
\end{proof}

Next, we apply Lemma~\ref{lem:sat_first} again to determine the saturation functions of the sequences of the form $a_1a_2a_3\dots a_{k-2}a_{k-1}a_k a_{k-1} a_{k-2}\dots a_2a_1$ up to a constant factor. 

\begin{cor}
	For all $n\ge k \ge 2$ and $u=a_1a_2a_3\dots a_{k-2}a_{k-1}a_k a_{k-1} a_{k-2}\dots a_2a_1$, we have $\sat(u,n)=\Theta(n)$.
\end{cor}
\begin{proof}
	The lower bound $\sat(u, n) = \Omega(n)$ follows from Lemma~\ref{lem:sat_first}. Since $\ex(u, n) = O(n)$ \cite{KV}, we have $\sat(u, n) \le \ex(u, n) = O(n)$.
\end{proof}

By Lemma~\ref{lem:sat_first}, if every letter appears at least twice in $u$, then $\sat(u,n)\ge n$. Next, we generalize this lower bound to sequences where every letter appears at least a given number of times.
We start with a useful definition.

\begin{definition}
    Let $u$ be a sequence. Denote by $f_u(k)$ the frequency of letter $k$ in $u$. Define $m_u:=\min_k\{f_u(k):f_u(k)>0\}$, i.e., the minimum number of occurrences of any letter in $u$.
\end{definition}
\begin{lem}
    For any sequence $u$, if every letter of $u$ appears at least twice in $u$, i.e., $m_u\ge 2$, then $\sat(u,n)\ge (m_u-1)n$ for sufficiently large $n$.
\end{lem}
\begin{proof}
Let $x$ be a $u$-saturated sequence on $n$ letters. If some letter $a\in[n]$ is absent in $x$, then adding $a$ anywhere in $x$ does not create a subsequence isomorphic to $u$ since $a$ only appears once in $x$ after the addition. This contradicts with $x$ being $u$-saturated. Thus, every letter from $[n]$ is in $x$.

If every letter from $[n]$ appears at least $m_u-1$ times in $x$, then we are done.
Suppose that the least frequent letter $a\in[n]$ has less than $m_u-1$ occurrences in $x$, then there is no way of adding $a$ to $x$ without violating $\|u\|$-sparsity. So, in $x$ there are at most $\|u\|-2$ letters before the first $a$ and at most $\|u\|-2$ letters after the last $a$, and between consecutive appearances of $a$ there are at most $2\|u\|-3$ letters. Hence, the length of $x$ is bounded from above as
$$
|x|\le 2(\|u\|-2)+(f_x(a)-1)(2\|u\|-3)+f_x(a)=f_x(a)(2\|u\|-2)-1.
$$
Let $n\ge 2\|u\|-2$, then, since $a$ is the least frequent letter in $u$, we have  $|x|\ge nf_x(a)\ge (2\|u\|-2)f_x(a)$, a contradiction. Thus, every letter from $[n]$ appears at least $m_u-1$ times in $x$ and the result follows.
\end{proof}

The next lemma shows that the saturation function of any sequence $u$ is always either $O(1)$ or $\Omega(n)$. We conjecture that it is always either $O(1)$ or $\Theta(n)$.

\begin{lem}
For any sequence $u$,
either $\sat(u,n)\ge n$ for every positive integer $n$ or $\sat(u,n)=O(1)$.
\end{lem}

\begin{proof}
If there exists a positive integer $n_0$ such that $\sat(u,n_0)<n_0$, then assume that $x$ is a shortest $u$-saturated sequence on $n_0$ letters. By the Pigeonhole principle, some letter $a\in[n_0]$ is not in $x$, so $\sat(u,n)\le \sat(u,n_0)$ for all $n>n_0$.
\end{proof}

For most of the remainder of this section, the main goal is to demonstrate a linear versus constant dichotomy for the saturation functions of sequences with $2$ distinct letters. We prove several results to cover different cases of this dichotomy. The first result covers sequences in which the first $2$ letters are not the same and the last $2$ letters are not the same.

\begin{thm}\label{lasttwodiff}
For any sequence $u$ of length $\ell$ on $2$ distinct letters for which the first $2$ letters of $u$ are not the same and the last $2$ letters of $u$ are not the same, we have $\sat(u, n) \le 2 \ell n$. 
\end{thm}

\begin{proof}
Given a sequence $u$ on $2$ distinct letters with the first $2$ letters of $u$ not the same and last $2$ letters not the same, we define a $u$-saturated sequence $s_2(u)$ on $2$ distinct letters $a,b$ which is simply the longest alternation $abab\dots$ that avoids $u$. Clearly $s_2(u)$ is u-saturated, since $s_2(u)$ avoids $u$ and adding any new letter while maintaining $2$-sparsity produces a longer alternation, which by definition must contain $u$. Note that $s_2(u)$ has length at most $2\ell$, where $\ell$ is the length of $u$. 

In order to produce a $u$-saturated sequence on $n$ distinct letters, we proceed by induction. Let $s$ be a $u$-saturated sequence on an alphabet $A$ of $n-1$ distinct letters, let $x$ be the last letter of $s$, and let $y$ be a letter which does not occur in $s$. We append to the end of $s$ a copy of $s_2(u)$ on the letters $x,y$, minus the first letter. In particular, the terminal segment of the resulting sequence $s’$ is a copy of $s_2(u)$, where the first letter of the copy is the last letter of $s$. 

We claim that $s’$ is a $u$-saturated sequence on the alphabet $A \cup \left\{y\right\}$. First, note that $s’$ avoids $u$. By definition of $s_2(u)$ and the fact that the first $2$ letters of $u$ are not the same, there cannot be a copy of $u$ on the letters $x,y$. Moreover, if there was a copy of $u$ on $y$ and one of the other letters not equal to $x$, then $y$ would only occur once in the copy since the last $2$ letters of $u$ are not the same. Thus, $y$ could be replaced with the last letter of s (which is $x$) to produce a copy of $u$ in $s$, a contradiction. Finally, there can be no copy of $u$ on $x$ and one of the other letters not equal to $y$, or else there would be a copy of $u$ in $s$, since the last $2$ letters of $u$ are different. Thus, $s’$ avoids $u$.

Now, suppose that we add a new letter from $A \cup \left\{y\right\}$ to $s’$ in a way that maintains $2$-sparsity. If the new letter is $x$, then adding the new letter would introduce a copy of $u$ in $s'$, either in the copy of $s$ or the copy of $s_2(u)$ depending respectively on whether the new letter is added before or after the last letter in the copy of $s$. If the new letter is $y$, then adding it anywhere in $s'$ while maintaining $2$-sparsity would introduce a copy of $u$, by $u$-saturation of $s_2(u)$. If the new letter is not equal to $x$ or $y$, then adding it anywhere in $s'$ while maintaining $2$-sparsity would introduce a copy of $u$, by $u$-saturation of $s$. Thus, $s'$ is $u$-saturated, and it has length at most $2 \ell n$, so we have $\sat(u, n) \le 2 \ell n$, where $\ell$ is the length of $u$. 
\end{proof}

Next, we prove a useful lemma which allows us to bound the effect of doubling the last letter of a sequence on its saturation function.

\begin{lem}\label{lemma:on}
Let $k<r$ be nonnegative integers and suppose that the total number of occurrences of some $n-k$ letters of an $r$-sparse sequence $x$ on $n$ letters is $\ell$. Then
$$
|x|\le\frac{r\ell+rk-k}{r-k}.
$$
\end{lem}
\begin{proof}
Without loss of generality assume the alphabet is $a_1,a_2,\dots,a_k,a_{k+1},\dots,a_n$ where the total number of occurrences of $a_{k+1},a_{k+2},\dots,a_n$ is $\ell$. The number of occurrences of each of $a_1,a_2,\dots,a_k$ is at most $1+\frac{|x|-1}{r}$. So
$$
|x|\le\ell+k\left(1+\frac{|x|-1}{r}\right),
$$
or
$$
|x|(r-k)\le \ell r+kr-k.
$$
\end{proof}

Now, we prove a general result about the effect of doubling the last letter of a sequence on its saturation function. After proving this result, we are able to demonstrate the linear versus constant dichotomy for the saturation functions of sequences with $2$ distinct letters.

\begin{thm}\label{addlastrepeat}
If sequence $u’$ on $r$ distinct letters is obtained from sequence $u$ by appending a new occurrence of the last letter of $u$ to the end of $u$, then $\sat(u’,n) < \sat(u,n)+rn$.
\end{thm}

\begin{proof}
Let $s$ be a $u$-saturated sequence of minimal length on an alphabet of size $n$. Let $A = \left\{a_1,a_2,\dots,a_n\right\}$ be the alphabet of $s$, suppose without loss of generality that the last $r-1$ letters of $s$ are $a_{n-r+2},\dots,a_n$ in order, and let $s’$ be obtained from $s$ by adding $a_1,a_2,\dots,a_n$ to the end. We claim that $s’$ avoids $u’$. Indeed, if $s’$ contained a copy of $u’$, then removing the last letter would produce a copy of $u$ in $s$, since the last $2$ letters of $u’$ are the same. 

Moreover, we claim that it is impossible to add any letter $a_1,a_2,\dots,a_{n-r+1}$ to $s’$ without producing a copy of $u’$ or violating $r$-sparsity. If we add an occurrence of the letter $a_i$ before the end of the copy of $s$ in $s’$ without violating $r$-sparsity, then it produces a copy of $u$ in the copy of $s$ by definition, which can be extended to a copy of $u’$ by using one of the last $n$ letters of $s’$. For $i\le n-r+1$, we can append $a_i$ to $s$ without violating $r$-sparsity. So, if we add an occurrence of $a_i$ with $i \le n-r+1$ after the end of the copy of $s$ in $s’$, then there are $2$ occurrences of $a_i$ after the end of the copy of $s$ in $s’$. The first occurrence produces a copy of $u$ with last letter $a_i$ by $u$-saturation, and the second occurrence produces a copy of $u’$ since the last $2$ letters of $u’$ are the same. 

Let $s''$ be the sequence obtained from $s'$ by adding occurrences of $a_{n-r+2},\dots,a_n$ to $s'$ while maintaining $u'$-avoidance and $r$-sparsity until it is impossible to do so. Note that all the added letters are after $s$, because otherwise we would introduce a copy of $u'$ or violate $r$-sparsity. By Lemma~\ref{lemma:on} the subsequence $t$ of $s''$ after the copy of $s$ has length at most $r n-r+1$, because each of $a_1,\dots,a_{n-r+1}$ has exactly one occurrence in $t$.
Thus, $s''$ has length less than $\sat(u,n)+r n$ and it avoids $u’$ by definition. Moreover, we cannot add any of $a_{n-r+2},\dots,a_n$ to $s''$ without producing a copy of $u’$, by definition. Since we cannot add any occurrences of $a_1,a_2,\dots,a_{n-r+1}$ to $s’$ without producing a copy of $u’$, we cannot add any occurrences of $a_1,a_2,\dots,a_{n-r+1}$ to $s''$ without producing a copy of $u’$. Therefore, $s''$ is $u’$-saturated and we have $\sat(u’,n) \le \sat(u,n)+r n$.
\end{proof}

Using Theorem~\ref{addlastrepeat} and Theorem~\ref{lasttwodiff}, we obtain the following corollary.

\begin{cor}
If sequence $u$ of length $\ell$ has $2$ distinct letters, then $\sat(u,n) \le 2 \ell n$.
\end{cor}

\begin{proof}
We prove that $\sat(u, n) \le 2 \ell n$ for all sequences $u$ of length $\ell$ with $2$ distinct letters by induction on $\ell$. It is true for $\ell = 2$ since in this case we must have $u$ isomorphic to $ab$, so $\sat(u, n) = 1$. If $\ell > 2$, then suppose that we have $\sat(v, n) \le 2 \ell' n$ for all sequences $v$ on $2$ distinct letters of length $2 \le \ell' < \ell$. If the last $2$ letters of $u$ are the same, then let $u'$ be obtained from $u$ by removing the last letter. We have $\sat(u,n) \le \sat(u',n)+2n-2 < 2 \ell n$ by Theorem~\ref{addlastrepeat} and the inductive hypothesis. If the last $2$ letters of $u$ are different, then consider the first $2$ letters of $u$. If they are different, then we have $\sat(u, n) \le 2 \ell n$ by Theorem~\ref{lasttwodiff}. Otherwise if they are the same, then let $u''$ be obtained from $u$ by removing the first letter. Symmetry, Theorem~\ref{addlastrepeat}, and the inductive hypothesis imply that $\sat(u, n) \le \sat(u'',n)+2n-2 < 2 \ell n$.
\end{proof}

Finally, this implies that the sequence saturation function exhibits a linear versus constant dichotomy when the forbidden sequences have $2$ distinct letters.

\begin{cor}
For any sequence $u$ with $2$ distinct letters, we have $\sat(u,n) = \Theta(n)$ or $\sat(u, n) = O(1)$.
\end{cor}

It is an open problem to extend this dichotomy to sequences with more than $2$ distinct letters. We obtained the following lemma as a step toward this extension.

\begin{lem}\label{lem:second_too_far}
    For any sequence $u$ with $r\ge 3$ distinct letters where its first letter appears exactly once and
    the next $r-1$ letters appear more than once, 
    we have $\sat(u,n)\ge n$ for every $n\ge r$.
\end{lem}
\begin{proof}
Suppose to the contrary that $\sat(u,n_0)<n_0$ for some $n_0\ge r$. We write $u=fu'$ where $f$ is a letter that does not appear in subsequence $u'$. Let $x$ be a $u$-saturated sequence of length $\sat(u,n_0)$ on $n_0$ letters. Note that $|x|\ge r$ because otherwise after adding a new letter to $x$ the length of $x$ will be at most $r$ and less than $|u|$, and $x$ would not contain $u$. If we insert a new letter between $x_{r-1}$ and $x_r$, by definition a copy of $u$ is introduced and it must start with the inserted letter. Therefore, there is a copy of $u'$ between $x_r$ and the last letter of $x$, and the copy consists of exactly $r-1$ distinct letters. We claim that the first letter of $u'$ is $x_r$. If the first letter of $u'$ is after $x_r$, since by $r$-sparsity $x_1,x_2,\dots,x_r$ are all distinct one of them does not appear in $u'$ and $x$ already contained $u$. For $x$ to be $u$-free each of $x_1,x_2,\dots,x_{r-1}$ appear in $u'$, so there are at most $r-2$ distinct letters among $x_1,x_2,\dots,x_{r-1}$, violating $r$-sparsity of $x$. Therefore the assumption that $\sat(u,n_0)<n_0$ for some $n_0\ge r$ is false, and $\sat(u,n)\ge n$ for every $n\ge r$.
\end{proof}
\begin{cor}
Suppose that $u$ is a sequence of length at least $3$. 
If less than $3$ distinct letters of $u$ appear just once, then $\sat(u,n)\ge n$.
\end{cor}
\begin{proof}
       If the first or last letter of $u$ appears multiple times, then by Lemma~\ref{lem:sat_first} we have $\sat(u,n)\ge n$. Otherwise suppose that both the first and last letters of $u$ appear just once, and all other letters of $u$ appear multiple times. Since $|u|\ge3$, there is at least one letter appearing more than once, so the length of $u$ is greater than the number of distinct letters of $u$. By Lemma~\ref{lem:second_too_far}, we have $\sat(u,n)\ge n$.
\end{proof}

\begin{cor}
Suppose that $u$ is a sequence of length at least $3$. 
If $\sat(u,n)=O(1)$, then $u$ has at least $3$ distinct letters that appear just once.
\end{cor}
\section{Semisaturation of sequences}\label{s:ssat}

As with the saturation function of sequences, we start with some simple results about semisaturation. 

\begin{lem}
    If $u$ has length $r$ and every letter occurs once in $u$, then $\ssat(u, n) = \min(n, r-1)$. 
\end{lem}

\begin{proof}
    We have $\ssat(u, n) \le \min(n, r-1)$ since $\ssat(u, n) \le \sat(u, n) = \min(n, r-1)$. To see that $\ssat(u, n) \ge \min(n, r-1)$, we suppose that $s$ is a $u$-semisaturated sequence on an alphabet of size $n$ and split into two cases. If $n \le r-1$, then every letter in the alphabet must occur in $s$, since any missing letter could be added to the beginning of $s$ without violating $r$-sparsity or inducing a new copy of $u$. If $n \ge r$, then $s$ must have length at least $r-1$, or else we could add a missing letter to the beginning of $s$ without violating $r$-sparsity, and the resulting sequence would still be too short to contain any copy of $u$.
\end{proof}

\begin{lem}\label{lem:ssat_first}
    Suppose that $u$ is a sequence in which the first letter occurs multiple times or the last letter occurs multiple times. Then $\ssat(u, n) \ge n$.
\end{lem}

\begin{proof}
    The proof is analogous to Lemma~\ref{lem:sat_first}.
\end{proof}

\begin{cor}
    For all $n \ge 1$ and $k \ge 1$, we have $\ssat(a_1 a_2 \dots a_k a_1, n) = n$. 
\end{cor}

\begin{proof}
    We have $\ssat(a_1 a_2 \dots a_k a_1, n) \le n$, since \[\ssat(a_1 a_2 \dots a_k a_1, n) \le \sat(a_1 a_2 \dots a_k a_1, n) = n.\] The lower bound $\ssat(a_1 a_2 \dots a_k a_1, n) \ge n$ follows from Lemma~\ref{lem:ssat_first}.
\end{proof}

For the following proofs, we introduce the notation $\up(n,\ell)$ to refer to the sequence of length $\ell n$ obtained by repeating $a_1 a_2 \dots a_n$ a total of $\ell$ times. We refer to the successive copies of $a_1 a_2 \dots a_n$ in $\up(n, \ell)$ as blocks $1, 2, \dots, \ell$. 

\begin{thm}
    For all sequences $u$, we have $\ssat(u, n) = O(n)$.
\end{thm}

\begin{proof}
    We construct a sequence on $n$ distinct letters which is $u$-semisaturated. Indeed, if $u$ has length $\ell$, consider the sequence $\up(n, \ell)$. Suppose that we add a new occurrence of letter $a_j$ in $\up(n, \ell)$.
    We say that we added a letter to block $i$, if we added it between two letters of block $i$ or next to the first or last letter of block $i$. A letter may be added to multiple blocks at the same time, but this is not a problem for the rest of the proof.
    Suppose that it is added to block $k$ for some $1 \le k \le \ell$. Then the new occurrence of $a_j$ can be used to make a new copy of $u$ by using the new occurrence of $a_j$ as the $k^{\text{th}}$ letter of the copy of $u$ and picking the $i^{\text{th}}$ letter of the copy of $u$ from the $i^{\text{th}}$ block. 
    \end{proof}

\begin{lem}
    If the first letter and the last letter of $u$ each occur exactly once, then $\ssat(u, n) = O(1)$.
\end{lem}

\begin{proof}
    Let $\ell$ be the length of $u$ and let $r$ be the number of distinct letters in $u$. We claim that $\up(r, 2\ell)$ is $u$-semisaturated on the alphabet $a_1, a_2, \dots, a_n$. Indeed, suppose that we insert some letter $a_j$ anywhere in $\up(r, 2\ell)$. If we insert it in the first half of $\up(r, 2\ell)$, then we have a new copy of $u$ with the new letter as the first letter of the copy and each of the next $l-1$ blocks contributing one letter to the copy. If we insert it in the second half or at the midpoint of $\up(r, 2\ell)$, then we have a new copy of $u$ with the new letter as the last letter of the copy and each of the preceding $l-1$ blocks contributing one letter to the copy.
\end{proof}

\begin{cor}\label{cor_char}
    For any sequence $u$, we have $\ssat(u, n) = O(1)$ if and only if the first letter and the last letter of $u$ each occur exactly once. Otherwise we have $\ssat(u, n) = \Theta(n)$.
\end{cor}

Recall that $u_s$ denotes an alternating sequence of length $s+2$.

\begin{lem}
For $n \ge 2$ and odd $s \ge 1$, we have $\ssat(u_s, n) \le \left(\frac{s+1}{2}\right)n+1$.
\end{lem}

\begin{proof}
    Consider the sequence obtained from \[\up\left(n, \frac{s+1}{2}\right) = \left(a_1 a_2 \dots a_n\right)^{\frac{s+1}{2}}\] by adding an occurrence of $a_1$ to the end. If we add a new letter $x \neq a_1$ to the beginning or end, then we induce a new copy of $u_s$ with first letter $x$ and second letter $a_1$. If we add a new letter $x \neq a_1$ to the interior, then we induce a new copy of $u_s$ with first letter $a_1$ and second letter $x$. If we add a new $a_1$ to the interior, then we induce a new copy of $u_s$ with first letter $a_1$ and second letter $a_2$.
\end{proof}

\begin{lem}
For $n \ge 2$ and even $s \ge 2$, we have $\ssat(u_s, n) \le \left(\frac{s+2}{2}\right)n$.
\end{lem}

\begin{proof}
    Consider the sequence \[\up\left(n, \frac{s+2}{2}\right) = \left(a_1 a_2 \dots a_n\right)^{\frac{s+2}{2}}.\] If we add a new letter $x \neq a_1$ to the beginning, then we induce a new copy of $u_s$ with first letter $x$ and second letter $a_1$. If we add a new letter $x \neq a_1$ anywhere in the sequence except the beginning, then we induce a new copy of $u_s$ with first letter $a_1$ and second letter $x$. If we add a new $a_1$ anywhere in the sequence except the end, then we induce a new copy of $u_s$ with first letter $a_1$ and second letter $a_n$. If we add a new $a_1$ to the end, then we induce a new copy of $u_s$ with first letter $a_2$ and second letter $a_1$.
\end{proof}

In the following lemma, we obtain a lower bound for the semisaturation function of alternating sequences that nearly matches our upper bound.

\begin{lem}
For all $n \ge 3$ and $s \ge 1$, we have $\ssat(u_s,n)\ge n\lfloor\frac{s}{2}\rfloor+1$ if $s$ is even and $\ssat(u_s,n)\ge n\lfloor\frac{s}{2}\rfloor+3$ if $s$ is odd.
\end{lem}

\begin{proof}
    The proof is identical to Lemma~\ref{lem:altlower}.
\end{proof}

Below we list some of our conjectures about semisaturation for alternating sequences. 

\begin{conj}\label{c:ssat_alt}
Suppose that $x=(x_1,x_2,\dots)$ is a shortest $u_s$-semisaturated sequence on $n$ letters.
\begin{enumerate}
\item $x$ has length $n\left(\frac{s+2}{2}\right)$ if $s$ is even and $n\left(\frac{s+1}{2}\right)+1$ if $s$ is odd.
\item If $s$ is even, then every letter occurs $\frac{s+2}{2}$ times in $x$.
\item If $s$ is odd, then one letter occurs $\frac{s+1}{2}+1$ times in $x$ and all other letters occur $\frac{s+1}{2}$ times in $x$.
\end{enumerate}
\end{conj}

In the table below, we have used a program to generate some $u_s$-semisaturated sequences which we conjecture to have minimum length.\\

\begin{tabular}{|c|c|l|}
	\hline
	s & n & some $u_s$-semisaturated sequences likely to be the shortest\\
	\hline
	2 & 6 & 0, 1, 0, 1, 2, 3, 2, 3, 4, 5, 4, 5 \\
	\hline
        2 & 6 & 0, 1, 2, 3, 0, 3, 2, 4, 5, 1, 4, 5 \\
	\hline
        2 & 6 & 0, 1, 2, 3, 4, 1, 4, 0, 3, 5, 2, 5 \\
	\hline
        2 & 6 & 0, 1, 2, 3, 4, 5, 2, 1, 5, 3, 0, 4 \\
	\hline
        4 & 4 & 0, 1, 0, 1, 0, 1, 2, 3, 2, 3, 2, 3 \\
	\hline
        4 & 4 & 0, 1, 2, 3, 1, 0, 2, 3, 1, 3, 0, 2 \\
	\hline
        4 & 4 & 0, 1, 2, 3, 1, 0, 3, 2, 1, 3, 0, 2 \\
	\hline
        4 & 4 & 0, 1, 2, 3, 2, 3, 2, 3, 0, 1, 0, 1 \\
	\hline
        3 & 5 & 0, 1, 2, 0, 1, 2, 0, 3, 4, 3, 4 \\
	\hline
        3 & 5 & 0, 1, 2, 3, 4, 2, 0, 1, 3, 2, 4 \\
	\hline
        3 & 5 & 0, 1, 2, 3, 4, 3, 1, 2, 4, 3, 0 \\
	\hline
        3 & 6 & 0, 1, 2, 3, 4, 5, 4, 0, 1, 2, 5, 4, 3 \\
        \hline
        5 & 4 & 0, 1, 0, 1, 0, 1, 0, 2, 3, 2, 3, 2, 3 \\
        \hline
        5 & 4 & 0, 1, 2, 3, 1, 0, 2, 3, 1, 0, 2, 3, 1 \\
        \hline
        5 & 4 & 0, 1, 2, 3, 1, 0, 2, 3, 1, 2, 3, 1, 0 \\
        \hline
        5 & 4 & 0, 1, 2, 3, 2, 0, 1, 3, 2, 0, 1, 3, 2 \\
        \hline
\end{tabular}

\section{Conclusion}\label{s:conc}

In this paper, we introduced saturation and semisaturation functions of sequences and proved a number of results about these functions. In particular, for alternating sequences we bounded both the saturation function and the semisaturation function up to a constant multiplicative factor. We conjecture that our upper bounds are sharp in both cases, i.e., we conjecture that we have constructed sequences of minimal length that are saturated/semisaturated for alternations. 

Perhaps the biggest remaining open problem about saturation functions of sequences is to demonstrate the following dichotomy.
\begin{conj}
For every sequence $u$, either $\sat(u,n)=\Theta(n)$ or $\sat(u,n)=O(1)$. \end{conj}
We have proved this dichotomy for all sequences $u$ with $2$ distinct letters, but it is open more generally. In order to prove the dichotomy in general, it suffices to prove that $\sat(u, n) = O(n)$ for all sequences $u$, since we already proved that $\sat(u, n) \ge n$ for all positive integers $n$ or $\sat(u, n) = O(1)$ for every sequence $u$. If this dichotomy is proved, a more ambitious open problem is to characterize all sequences $u$ as to whether $\sat(u, n) = \Theta(n)$ or $\sat(u, n) = O(1)$. We already proved the analogous characterization for semisaturation functions of sequences in Corollary~\ref{cor_char}.

Besides the open problems that we have discussed, there are some additional future directions for research on saturation functions and semisaturation functions of sequences. In a number of papers about extremal functions of Davenport-Schinzel sequences (see, e.g., \cite{nivasch, pettie15, geneson15}), researchers have investigated the function $\psi_s(m, n)$, which is the maximum possible length of any Davenport–Schinzel sequence of order $s$ on $n$ letters that can be partitioned into at most $m$ contiguous blocks, where all letters in each block are distinct. Bounds on $\psi_s(m, n)$ were used to determine bounds on Davenport-Schinzel sequences without the additional block restriction. It would be interesting to investigate saturation functions and semisaturation functions for sequences with restrictions on the number of blocks.\\

\noindent {\bf Acknowledgement.} 
We are thankful for the valuable comments from the reviewers, which help improve the article significantly.
This research started in CrowdMath 2021. We thank the organizers for this opportunity. We also thank Art of Problem Solving user Anurag Ramachandran for proving Lemma~\ref{lem:friends} and collaborating on the proof of Lemma~\ref{lemma:construction}.
Shen-Fu Tsai is supported by the Ministry of Science and Technology of Taiwan under grants MOST 111-2115-M-008-010-MY2 and MOST 113-2115-M-008-006-MY3.

\end{document}